\documentclass[a4paper,11pt]{article}
\usepackage[utf8]{inputenc}
\usepackage{a4wide}
\usepackage{latexsym,amsfonts,amsmath,amssymb,mathrsfs,url,amsthm}
\usepackage{mathtools}
\usepackage{color,graphicx}
\newtheorem{theorem}{Theorem}
\newtheorem{remark}[theorem]{Remark}
\newtheorem{lemma}[theorem]{Lemma}

\providecommand{\keywords}[1]
{
  \noindent \small	
  \textbf{Keywords:} #1
}
\providecommand{\amscode}[1]
{
  \noindent \small	
  \textbf{AMS subject classifications:} #1
}
\newcommand{\bsx}{\boldsymbol{x}}

\newcommand{\RR}{\mathbb{R}}
\newcommand{\ZZ}{\mathbb{Z}}
\newcommand{\Scal}{\mathcal{S}}
\newcommand{\Xcal}{\mathcal{X}}
\allowdisplaybreaks

\title{One-dimensional quasi-uniform Kronecker sequences\thanks{The work of the author is supported by JSPS KAKENHI Grant Number 23K03210.}}
\author{Takashi Goda\thanks{School of Engineering, The University of Tokyo, 7-3-1 Hongo, Bunkyo-ku, Tokyo 113-8686, Japan ({\tt goda@frcer.t.u-tokyo.ac.jp}).}}
\date{\today}

\begin{document}

\maketitle

\begin{abstract}
    In this short note, we prove that the one-dimensional Kronecker sequence $i\alpha \bmod 1, i=0,1,2,\ldots,$ is quasi-uniform if and only if $\alpha$ is a badly approximable number. Our elementary proof relies on a result on the three-gap theorem for Kronecker sequences due to Halton (1965).
\end{abstract}

\amscode{11K36, 11J71, 65D12}
\section{Introduction and the main result}
Motivated primarily by applications to scattered data approximation \cite{wendland2005scattered}, we are concerned with the quasi-uniformity of infinite sequences of points. Let $\Xcal$ be a compact convex subset of $\RR^d$ for a dimension $d$ with $\mathrm{vol}(\Xcal)>0$. For an infinite sequence of points $\Scal=(\bsx_i)_{i=0,1,2,\ldots}$ in $\Xcal$, we say that $\Scal$ is \emph{quasi-uniform} over $\Xcal$ if there exists a constant $C>1$ such that
\[ \rho_n(\Scal):=\frac{h_n(\Scal)}{q_n(\Scal)}\leq C\]
for all positive integers $n$, where $h_n(\Scal)$ and $q_n(\Scal)$ denote the fill distance and the separation radius of the first $n$ points of $\Scal$, respectively, i.e.,
\[ h_n(\Scal):=\sup_{\bsx\in \Xcal}\min_{0\leq i<n}\|\bsx-\bsx_i\|\quad \text{and}\quad q_n(\Scal):=\frac{1}{2}\min_{0\leq i<j<n}\|\bsx_i-\bsx_j\|,\]
with $\|\cdot\|$ representing the Euclidean norm on $\RR^d$. Note that $\rho_n(\Scal)$ is known as the \emph{mesh ratio} of the first $n$ points of $\Scal$. From the definition, it is evident that $h_n(\Scal)\geq q_n(\Scal)$, and thus $\rho_n(\Scal)\geq 1$. Quasi-uniformity refers to the property of infinite sequences in which successive points fill the domain $\Xcal$ evenly while keeping some distances from each other. This property is desirable for sampling nodes for function evaluation in scattered data approximation \cite{narcowich2006sobolev,teckentrup2020convergence,wynne2021convergence}. 

Another popular measure of uniformity is the discrepancy, which has been widely studied in the field of quasi-Monte Carlo (QMC) methods \cite{drmotatichy1997sequences,niederreiter1992random}. In fact, the fill distance $h_n(\Scal)$ has also been studied in this context, often referred to as ``dispersion'' instead of ``fill distance''. A sequence $\Scal$ in the unit cube $[0,1]^d$ is said to be a low-dispersion sequence if $\limsup_{n\to \infty}n^{1/d}h_n(\Scal)<\infty$. Many explicit constructions of low-discrepancy sequences have already been proven to exhibit low dispersion as well \cite[Chapter~6]{niederreiter1992random}. However, it remains largely unknown whether low-discrepancy or low-dispersion sequences are also quasi-uniform. While a recent paper by the author \cite{goda2024sobol} has proven that the Sobol' sequence \cite{sobol67distribution}, one of the best-known low-discrepancy sequences, is not quasi-uniform in dimension $2$ over the unit square $[0,1]^2$, it is far from trivial whether higher-dimensional Sobol' sequences or other common low discrepancy sequences are quasi-uniform over the $d$-dimensional unit cube $[0,1]^d$. 

In this short note, we focus on the simplest one-dimensional case and investigate the quasi-uniformity of one of the well-known sequences, the Kronecker sequence, over the unit interval $[0,1]$. The Kronecker sequence is defined as
\begin{align}\label{eq:kronecker_seq}
x_i=i\alpha \bmod 1,\quad i=0,1,2,\ldots,
\end{align}
for a real number $\alpha>0$. Here and in what follows, we adopt the notation $x\bmod 1$ to denote the fractional part of $x$, which is defined as the unique number $x'\in [0,1)$ such that $x-x'$ is an integer. This differs from the convention where $x\bmod 1$ denotes the coset $x+\ZZ$ in the quotient group $\RR/\ZZ$. It is known that this sequence achieves the star discrepancy of $O((\log n)/n)$ if and only if $\alpha$ is irrational and its continued fraction expansion $[a_0;a_1,a_2,\ldots]$ satisfies $(1/m)\sum_{j=1}^{m}a_j$ being uniformly bounded above for all $m\geq 1$ \cite[Corollary~1.65]{drmotatichy1997sequences}. This order is best possible from a general lower bound on the star discrepancy by Schmidt \cite{schmidt1972irregularities}. Additionally, it has been shown in \cite{drobot1986dispersion,niederreiter1984measure} that the Kronecker sequence is also a low-dispersion sequence if and only if $\alpha$ is a badly approximable number, i.e., $\alpha=[a_0;a_1,a_2,\ldots]$ is irrational and $a_m$ is uniformly bounded above for all $m\geq 1$. Beyond that, Kronecker sequences have recently been studied in terms of (week) Poissonian pair correlations \cite{weiss2023remarks,weiss2022sequences}.

In this note, we prove the following result on the quasi-uniformity of Kronecker sequences. It turns out that the ``if and only if'' condition coincides with the one for the Kronecker sequence to be a low-dispersion sequence. As we will see, our elementary proof relies on a result on the three-gap theorem for Kronecker sequences due to Halton (1965).

\begin{theorem}\label{thm}
    The Kronecker sequence $\Scal=(x_i)_{i=0,1,2,\ldots}$, defined by \eqref{eq:kronecker_seq}, is quasi-uniform over $[0,1]$ if and only if $\alpha$ is a badly approximable number. 
\end{theorem}

\begin{remark}
    The van der Corput sequence, another well-known low-discrepancy sequence in dimension $1$, can be easily proven to be quasi-uniform in any base $b \geq 2$. In fact, the greedy packing algorithm from \cite{pronzato2023quasi} reproduces the van der Corput sequence in base $2$ by placing the first point at $1/2$, and the mesh ratio $\rho_n$ is bounded above by $2$ for any $n$.
\end{remark}

\section{Proof}
As introduced already, let us denote the continued fraction of $\alpha>0$ by $[a_0;a_1,a_2,\ldots]$, i.e.,
\[ \alpha=a_0+\frac{1}{\displaystyle a_1+\frac{1}{\displaystyle a_2+\frac{1}{\ddots}}}.\]
This continued fraction is finite if and only if $\alpha$ is rational. In such cases, the corresponding Kronecker sequence becomes periodic, with only finitely many distinct values appearing infinitely often. Consequently, there exists $n_0\geq 1$ such that the separation radius $q_n$ equals $0$ for all $n\geq n_0$, making the mesh ratio $\rho_n$ unbounded. We will now proceed under the assumption that $\alpha$ is irrational, implying an infinite continued fraction.

Let us introduce the following notation:
\[ \alpha_0=\alpha \bmod 1 \quad \text{and}\quad \alpha_{m+1}= 1/\alpha_{m}\bmod 1,\quad  m=0,1,2,\ldots. \]
In fact, it is easy to check that $\alpha_m=[0;a_{m+1},a_{m+2},\ldots]$ for all $m$. Furthermore, let
\[ \eta_{-1}=1\quad \text{and}\quad \eta_m=\prod_{j=0}^{m}\alpha_m,\quad  m=0,1,2,\ldots , \]
\[ s_0=1,\quad s_1=a_1, \quad \text{and}\quad s_{m+2}=a_{m+2}s_{m+1}+s_{m},\quad m=0,1,2,\ldots , \]
and
\[ n_0=1,\quad \text{and}\quad n_m=s_m+s_{m-1},\quad m=1,2,3,\ldots . \]
Then the following result was proven by Halton \cite[Corollary~2]{halton1965distribution} about the famous three-gap theorem for Kronecker sequences \cite{marklof2017three,sos1958distribution,swierczkowski1959successive,weiss2020deducing}.

\begin{lemma}[Halton]\label{lem:3-gap}
    Let $n_m\leq n< n_{m+1}$ for some $m\geq 0$, and consider the first $n$ points of the Kronecker sequence \eqref{eq:kronecker_seq} with an irrational $\alpha$. Rearrange these points in ascending order and denote them by $x^{(0)}\leq x^{(1)}\leq \cdots \leq x^{(n-1)}$ with $x^{(0)}=0$ and $x^{(n-1)}<1=x^{(n)}$. Let $n=n_m+hs_m+k$ with $0\leq h\leq a_{m+1}-1$ and $0\leq k\leq s_m-1$. Then, among the $n$ gap lengths $x^{(1)}-x^{(0)},x^{(2)}-x^{(1)},\ldots,x^{(n)}-x^{(n-1)}$, there exist at most three distinct values. In particular, $s_{m-1}+hs_m+k$ lengths are equal to $\eta_m$, $s_m-k$ lengths are equal to $\eta_{m-1}-h\eta_m$, and $k$ lengths are equal to $\eta_{m-1}-(h+1)\eta_m$.
\end{lemma}

Now we are ready to prove Theorem~\ref{thm}.
\begin{proof}[Proof of Theorem~\ref{thm}]
    First, we prove that the Kronecker sequence is quasi-uniform for any $\alpha$ being a badly approximable number. For some $m\geq 0$, let $n=n_m+hs_m+k$ with $0\leq h\leq a_{m+1}-1$ and $0\leq k\leq s_m-1$. For the three gap lengths given in Lemma~\ref{lem:3-gap}, we have $\eta_{m-1}-h\eta_m=\eta_m+\left( \eta_{m-1}-(h+1)\eta_m\right)$, i.e., one length is the sum of the other two. Therefore, for the separation radius, we have
    \begin{align*}
        q_n(\Scal) & =\frac{1}{2}\min_{0\leq i<j<n}|x_i-x_j| = \frac{1}{2}\min_{0\leq i<n-1}|x^{(i+1)}-x^{(i)}| \geq \frac{1}{2}\min_{0\leq i<n}|x^{(i+1)}-x^{(i)}|\\
        & = \frac{1}{2}\min\left\{ \eta_m, \eta_{m-1}-h\eta_m, \eta_{m-1}-(h+1)\eta_m\right\}=\frac{1}{2}\min\left\{ \eta_m, \eta_{m-1}-(h+1)\eta_m\right\}.
    \end{align*}
    Similarly, for the fill distance, we have
    \begin{align*}
        h_n(\Scal) & =\sup_{x\in [0,1]}\min_{0\leq i<n}|x-x_i|=\max\left\{\frac{1}{2}\max_{0\leq i<n-1}|x^{(i+1)}-x^{(i)}|, |x^{(n)}-x^{(n-1)}|\right\}\\
        & \leq \max_{0\leq i<n}|x^{(i+1)}-x^{(i)}|=\max\left\{ \eta_m, \eta_{m-1}-h\eta_m, \eta_{m-1}-(h+1)\eta_m\right\}=\eta_{m-1}-h\eta_m.
    \end{align*}
    These lead to an upper bound on the mesh ratio as
    \begin{align*}
        \rho_n(\Scal) & =\frac{h_n(\Scal)}{q_n(\Scal)} \leq \frac{2(\eta_{m-1}-h\eta_m)}{\min\left\{ \eta_m, \eta_{m-1}-(h+1)\eta_m\right\}}= \frac{2(1-h\alpha_m)}{\min(\alpha_m,1-(h+1)\alpha_m)}.
    \end{align*}
    Since it immediately follows from the definition of $\alpha_m$'s that
    \[ \frac{1}{a_{m+1}+1}<\alpha_m<\frac{1}{a_{m+1}+1/a_{m+2}}, \]
    the mesh ratio is further bounded above by
    \begin{align*}
        \rho_n(\Scal) & \leq 2\max\left\{ \frac{1-h\alpha_m}{\alpha_m}, \frac{1-h\alpha_m}{1-(h+1)\alpha_m}\right\} = 2\max\left\{ \frac{1}{\alpha_m}-h, \frac{1}{\displaystyle 1-\frac{1}{\displaystyle \frac{1}{\alpha_m}-h}}\right\}\\
        & \leq 2\max\left\{ a_{m+1}+1-h, \frac{1}{\displaystyle 1-\frac{1}{\displaystyle a_{m+1}+1/a_{m+2}-h}}\right\}\\
        & \leq 2\max\left\{ a_{m+1}+1, a_{m+2}+1\right\}\leq 2+2\sup_{j\geq 1}a_j.
    \end{align*}
    Thus, if $\alpha$ is a badly approximable number, the mesh ratio is bounded by a constant $C\leq 2+2\sup_{j\geq 1}a_j$ for all $n$, showing that the corresponding Kronecker sequence is quasi-uniform.

    Next, we prove that the Kronecker sequence is not quasi-uniform if $\alpha$ is not a badly approximable number. For this, let us consider the case where $n=n_m$ for $m\geq 3$, i.e., $n=n_m+hs_m+k$ with $h=k=0$. Then, it follows from Lemma~\ref{lem:3-gap} that there exist exactly two distinct values among the $n$ gap lengths, which are $\eta_m$ for $s_{m-1}$ gaps and $\eta_{m-1}$ for $s_{m}$ gaps. Since we assume $m\geq 3$, we have $s_{m}>s_{m-1}\geq 2$. Therefore, for the separation radius, we have
    \begin{align*}
        q_n(\Scal) & =\frac{1}{2}\min_{0\leq i<j<n}|x_i-x_j| = \frac{1}{2}\min_{0\leq i<n-1}|x^{(i+1)}-x^{(i)}| = \frac{1}{2}\min_{0\leq i<n}|x^{(i+1)}-x^{(i)}|\\
        & = \frac{1}{2}\min\left\{ \eta_m, \eta_{m-1}\right\}=\frac{1}{2}\eta_m.
    \end{align*}
    Similarly, for the fill distance, we have
    \begin{align*}
        h_n(\Scal) & =\sup_{x\in [0,1]}\min_{0\leq i<n}|x-x_i|=\max\left\{\frac{1}{2}\max_{0\leq i<n-1}|x^{(i+1)}-x^{(i)}|, |x^{(n)}-x^{(n-1)}|\right\}\\
        & \geq \frac{1}{2}\max_{0\leq i<n}|x^{(i+1)}-x^{(i)}|=\frac{1}{2}\max\left\{ \eta_m, \eta_{m-1}\right\}=\frac{1}{2}\eta_{m-1}.
    \end{align*}
    This leads to a lower bound on the mesh ratio as
    \begin{align*}
        \rho_n(\Scal) & =\frac{h_n(\Scal)}{q_n(\Scal)} \geq \frac{\eta_{m-1}}{\eta_m}= \frac{1}{\alpha_m}\geq a_{m+1}+\frac{1}{a_{m+2}}.
    \end{align*}
    Thus, we have shown that the mesh ratio cannot be uniformly bounded above in all $n$ if $a_m$'s are not uniformly bounded above, implying that the corresponding Kronecker sequence is not quasi-uniform.
\end{proof}

\section*{Acknowledgment}
The author would like to thank Kosuke Suzuki for discussions and comments.

\bibliographystyle{plain}
\bibliography{ref.bib}

\begin{thebibliography}{10}

\bibitem{drmotatichy1997sequences}
M.~Drmota and R.~F. Tichy.
\newblock {\em Sequences, Discrepancies and Applications}.
\newblock Springer, 1997.

\bibitem{drobot1986dispersion}
V.~Drobot.
\newblock {On dispersion and Markov constants}.
\newblock {\em Acta Math. Hung.}, 47:89--93, 1986.

\bibitem{goda2024sobol}
T.~Goda.
\newblock {The Sobol' sequence is not quasi-uniform in dimension 2}.
\newblock {\em Proc. Amer. Math. Soc.}, to appear, 2024.

\bibitem{halton1965distribution}
J.~H. Halton.
\newblock The distribution of the sequence $\{n\xi\}\, (n= 0, 1, 2,…)$.
\newblock {\em Proc. Camb. Phil. Soc.}, 61(3):665--670, 1965.

\bibitem{marklof2017three}
J.~Marklof and A.~Str{\"o}mbergsson.
\newblock The three gap theorem and the space of lattices.
\newblock {\em Amer. Math. Monthly}, 124(8):741--745, 2017.

\bibitem{narcowich2006sobolev}
F.~J. Narcowich, J.~D. Ward, and H.~Wendland.
\newblock {Sobolev error estimates and a Bernstein inequality for scattered data interpolation via radial basis functions}.
\newblock {\em Constr. Approx.}, 24:175--186, 2006.

\bibitem{niederreiter1984measure}
H.~Niederreiter.
\newblock On a measure of denseness for sequences.
\newblock In {\em Topics in Classical Number Theory, Vol.~II}, Colloq. Math. Soc. J\'{a}nos Bolyai 34, pages 1163--1208. North-Holland, Amsterdam, 1984.

\bibitem{niederreiter1992random}
H.~Niederreiter.
\newblock {\em Random Number Generation and Quasi-Monte Carlo Methods}.
\newblock SIAM, 1992.

\bibitem{pronzato2023quasi}
L.~Pronzato and A.~Zhigljavsky.
\newblock Quasi-uniform designs with optimal and near-optimal uniformity constant.
\newblock {\em J. Approx. Theory}, 294:105931, 2023.

\bibitem{schmidt1972irregularities}
W.~Schmidt.
\newblock {Irregularities of distribution, VII}.
\newblock {\em Acta Arith.}, 21(1):45--50, 1972.

\bibitem{sobol67distribution}
I.~M. Sobol'.
\newblock On the distribution of points in a cube and the approximate evaluation of integrals.
\newblock {\em Zh. Vychisl. Mat. Mat. Fiz.}, 7(4):784--802, 1967.

\bibitem{sos1958distribution}
V.~T. S{\'o}s.
\newblock On the distribution mod 1 of the sequence $n\alpha$.
\newblock {\em Ann. Univ. Sci. Budapest, Eötvös Sect. Math.}, 1:127--134, 1958.

\bibitem{swierczkowski1959successive}
S.~\'{S}wierczkowski.
\newblock On successive settings of an arc on the circumference of a circle.
\newblock {\em Fund. Math.}, 46(2):187--189, 1959.

\bibitem{teckentrup2020convergence}
A.~L. Teckentrup.
\newblock {Convergence of Gaussian process regression with estimated hyper-parameters and applications in Bayesian inverse problems}.
\newblock {\em SIAM/ASA J. Uncertain. Quantif.}, 8(4):1310--1337, 2020.

\bibitem{weiss2020deducing}
C.~Weiss.
\newblock {Deducing three gap theorem from Rauzy-Veech induction}.
\newblock {\em Rev. Colomb. Mat.}, 54(1):31--37, 2020.

\bibitem{weiss2023remarks}
C.~Weiss.
\newblock {Remarks on the pair correlation statistic of Kronecker sequences and lattice point counting}.
\newblock {\em Beitr. Algebra Geom.}, 2023.

\bibitem{weiss2022sequences}
C.~Weiss and T.~Skill.
\newblock {Sequences with almost Poissonian pair correlations}.
\newblock {\em J. Number Theory}, 236:116--127, 2022.

\bibitem{wendland2005scattered}
H.~Wendland.
\newblock {\em Scattered Data Approximation}.
\newblock Cambridge University Press, 2005.

\bibitem{wynne2021convergence}
G.~Wynne, F.-X. Briol, and M.~Girolami.
\newblock {Convergence guarantees for Gaussian process means with misspecified likelihoods and smoothness}.
\newblock {\em J. Mach. Learn. Res.}, 22(1):5468--5507, 2021.

\end{thebibliography}

\end{document}